\documentclass[a4paper,12pt,reqno]{amsart}
\usepackage{a4}
\usepackage{eucal}
\usepackage{enumerate}
\usepackage[T1]{fontenc}
\usepackage[utf8]{inputenc}
\usepackage{mathtools}
\usepackage{url}
\usepackage{color}
\usepackage[all]{xy}
\usepackage{amsmath,amsfonts,amssymb}
\usepackage{faktor}
\usepackage{cleveref}
\usepackage{ulem}

\usepackage{wasysym}
\usepackage{stmaryrd}

\newcommand\nat{\mathbb N}
\newcommand\zz{\mathbb Z}

\newcommand\rr{\mathbb R}

\newcommand{\mg}[1]{{#1}^{\times}}
\newcommand{\sq}[1]{{#1}^{\times 2}}
\newcommand{\sums}[1]{\left(\mbox{$\mathsf{\Sigma}{#1}^{2}$}\right)^\times}
\newcommand{\mfm}{\mathfrak{m}}
\newcommand{\s}{\sigma}

\newcommand\mc\mathcal

\swapnumbers
\newtheorem{thm}{Theorem}[section]
\newtheorem{lem}[thm]{Lemma}
\newtheorem{prop}[thm]{Proposition}
\newtheorem{cor}[thm]{Corollary}

\theoremstyle{definition}

\newtheorem{ex}[thm]{Example}
\newtheorem{rem}[thm]{Remark}
\numberwithin{equation}{thm}

\newcommand{\ovl}{\overline}

\newcommand{\car}{\mathsf{char}}

\newcommand{\half}{\hbox{$\frac{1}2$}}
\renewcommand{\min}{\mathsf{min}}
\renewcommand{\max}{\mathsf{max}}

\renewcommand{\setminus}{\smallsetminus}

\renewcommand{\leq}{\leqslant}
\renewcommand{\geq}{\geqslant}
\newcommand{\gen}{\mathfrak{g}}

\DeclareMathOperator{\Spec}{\mathsf{Spec}}

\renewcommand{\log}{\mathsf{log}}
\renewcommand{\deg}{\mathsf{deg}}
\newcommand{\G}{\mathsf{G}}
\newcommand{\alg}{\mathsf{alg}}
\newcommand{\sep}{\mathsf{sep}}
\newcommand{\per}{\mathsf{per}}

\renewcommand{\dim}{\mathsf{dim}}

\renewcommand{\cref}{\Cref}

\title{Nonsplit conics in the reduction of an arithmetic curve}

\date{26.11.2023}
\author{Karim Johannes Becher}
\address{University of Antwerp, Department of Mathematics, Middelheimlaan 1, 2020 Ant\-werp, Belgium}
\email{karimjohannes.becher@uantwerpen.be}
\author{David Grimm}
\address{Universidad de Santiago de Chile, Departamento de Matem\'atica y Ciencia de la Computaci\'on, Av. Libertador Bernardo O' Higgins 3363, Santiago de Chile, Chile}
\email{david.grimm@usach.cl}

\thanks{This work was supported by the \textit{FWO Odysseus Programme} (project G0E6114N, \textit{Explicit Methods in Quadratic Form Theory}), funded by the \textit{Fonds Wetenschappelijk Onderzoek -- Vlaanderen},  by \textit{ANID} (proyecto FONDECYT 11150956) 
and by the \textit{Universidad de Santiago de Chile} (USACH proyecto DICYT, Codigo 041933G)}

\begin{document}
\begin{abstract}\noindent
For a function field in one variable $F/K$ and a discrete valuation $v$ of $K$ with perfect residue field $k$, we bound the number of discrete valuations on $F$ extending $v$ whose residue fields are non-ruled function fields in one variable  over $k$.
Assuming that $K$ is relatively algebraically closed in $F$, we find that the number of non-ruled residually transcendental extensions of $v$ to $F$ is bounded by $\gen+1$ where $\gen$ is the genus of $F/K$. 
An application to sums of squares in function fields of curves over $\rr(\!(t)\!)$ is outlined.

\medskip

\medskip

\noindent
\textit{Classification (MSC 2020):} 12D15, 12J10, 12J25, 14H05, 14H25
\end{abstract}

\maketitle

%%%%%%%%%%%%%%
\section{Introduction} %%%
%%%%%%%%%%%%%%

A finitely generated field extension of transcendence degree one is called a \textit{function field in one variable}.
Let $K$ be a field and let $v_0$ be a non-trivial valuation on $K$ whose residue field we denote by $k$.
Let $F/K$ be a function field in one variable, $v$ a valuation on $F$ extending $v_0$ and $\kappa_v$ the residue field of $v$.
We call $v$ \textit{residually transcendental over $K$} if the residue field extension $\kappa_v/k$ is transcendental; this implies that the valuation is non-trivial on $K$ and 
that $\kappa_v/k$ is a function field in one variable.

The function field $F/K$ is called \textit{ruled} if $F$ is a rational function field over a finite field extension of $K$.
Let $v$ be a residually transcendental extension of $v_0$ to $F$. We say that $v$  is \textit{ruled} (\textit{over $K$}) if the residue field extension $\kappa_v/k$ is ruled. 
The Ruled Residue Theorem by J.~Ohm \cite{Ohm} asserts that, if $F/K$ is ruled, then every residually transcendental extension of $v_0$ to $F$ is ruled.
(This answered a question by Nagata in~\cite{Nag}, where the statement was obtained for the special case  where $v_0$ can be obtained as a composition of discrete valuations.)
More generally one may ask for a description of the set of non-ruled residually transcendental extensions of $v_0$ to $F$. 
In particular one may ask whether this set is finite.

Recent work has revealed for a few special types of function fields in one variable $F/K$ that a valuation on $K$ with good residue characteristic admits at most one non-ruled residually transcendental extension to $F$: 
This was obtained in \cite{BeGu19} for function fields of conics, in \cite{BeGuMi23} for function fields of elliptic curves and in \cite[Cor.~1.3]{Dut23} for function fields of Fermat type of prime degree.

This motivates the question whether there is a general bound on the number of such valuation extensions in terms of the genus of $F/K$.
A partial positive answer to this question follows from a genus inequality, which was  proven  in \cite{Mathieu} for the case where $v_0$ is discrete
and then refined and extended to the general case in \cite{Mat87} and \cite{GMP89}.
This inequality gives a finite upper bound in terms of the genus of $F/K$ on the number of residually transcendental extensions of $v_0$ to $F$ whose residue field extensions $\kappa_v/k$ have positive genus.
Therefore the main problem is to determine those residually transcendental extensions of $v_0$ to $F$ for which the residue field extension $\kappa_v/k$ is of genus zero but not ruled.

In this article we study the case where $v_0$ is a discrete valuation of rank one. 
Our main result, \Cref{T:main-bound}, assumes that the residue field $k$ of $v_0$ is perfect and 
$K$ is relatively algebraically closed in $F$, and it states 
that the number of  non-ruled residually transcendental extensions of $v_0$ is bounded by $\gen(F/K)+1$ where $\gen(F/K)$ is the genus of $F/K$.
We expect that this result can be extended to the case where $v_0$ is an arbitrary valuation on $K$ and where $k$ is not necessarily perfect, but 
a proof of such a generalisation would certainly require different tools.
\smallskip

The research leading to the main result of this article was inspired by a question about sums of squares in the function field $F$ of a curve $C$
over the field $K=\rr(\!(t)\!)$.
Let $\sums{F}$ denote the subgroup of $\mg{F}$ consisting of the nonzero elements of $F$ which are sums of squares in $F$. Let further $\mg{(F^2+F^2)}$ denote the subgroup  of $\sums{F}$ given by the elements which are sums of two squares in $F$.
If $C=\mathbb{P}^1_K$, then 
  we have $\sums{F}=\mg{(F^2+F^2)}$, by a result of E.~Becker \cite[Chap.~III, Theorem~4]{Becker}.
For the curve $C:Y^2=(tX-1)(X^2+1)$ over $K$, it was shown in \cite[(3.11)]{TVGY} that $tX$ is a sum of three but not of two squares in $F$, whereby $\mg{(F^2+F^2)}\subsetneq \sums{F}$.
This observation led in \cite{BVG} to the study of the $2$-elementary abelian  group $$\G(F)=\sums{F}/\mg{(F^2+F^2)}\,.$$
For the case of a hyperelliptic curve $C:Y^2=f(X)$ over $K$ with $f(X)\in K[X]$ square-free, it was proven in \cite[Theorem~3.10]{BVG} that $$\log_2|\G(F)|\leq \hbox{$\frac{1}2$}(\deg(f)+1)\,.$$ 
In \cite[Theorem~6.12]{BGVG} it was shown that $\G(F)$ is finite for any function field in one variable $F/K$ for $K=\rr(\!(t)\!)$. This result is based on \cite[Theorem 6.11]{BGVG}, which yields that $\log_2|\G(F)|$ is equal to the number of $\zz$-valuations $v$ on $F$ such that $-1$ is a sum of two squares but not a square in the residue field of $v$.
Since any such valuation is a non-ruled residually transcendental extension of the $t$-adic valuation on $K=\rr(\!(t)\!)$, our main result (\Cref{T:main-bound}) yields the bound
$$\log_2|\G(F)|\leq \gen(F/K)+1\,.$$
Note that, in the case where $F$ is given by a hyperelliptic curve $Y^2=f(X)$ over $K$ with $f(X)\in K[X]$ square-free,
we have $\gen(F/K)=\lfloor \frac{\deg(f)-1}2\rfloor$, so we recover the above-mentioned bound from \cite[Theorem~3.10]{BVG}.
It follows from \cite[Example 4.5]{ManF23} that this bound is optimal:  For any $g\in\nat$, the function field $F$ of the curve
$Y^2=-\prod_{i=0}^g(X^2+t^{2i})$ is such that $\gen(F/K)=g$ and $\log_2|\G(F)|=g+1$.
(For $g=1$, this is \cite[Example 5.12]{BVG}.)
\medskip

Let us now fix some terminology for valuations on an arbitrary field $F$.
For a valuation $v$ on $F$ we denote by $\mc{O}_v$ the corresponding valuation ring, by $\mfm_v$ the maximal ideal of $\mc{O}_v$, and by $\kappa_v$ its residue field $\mc{O}_v/\mfm_v$.
If $v$ has value group $\zz$, then we call $v$ a \textit{$\zz$-valuation}, and in this case $\mc{O}_v$ is a discrete valuation ring of $F$. Any discrete valuation ring of $F$ is given in this way by a unique $\zz$-valuation on $F$.
We denote by $\Omega(F)$ the set of all $\zz$-valuations on $F$. 

Consider now a discrete valuation ring $T$ with field of fractions $K$ and residue field $k$.
For a function field in one variable $F/K$,
we set
\begin{eqnarray*}
\Omega_T(F) & = & \{v\in \Omega(F)\mid \mc{O}_v\cap K=T\mbox{ and $\kappa_v/k$ is transcendental}\,\}\\
\Omega_T^\ast(F) &=&\{v\in \Omega_T(F)\mid \kappa_v/k\mbox{ is not ruled}\,\}.
\end{eqnarray*}
Under the assumption that $T$ is complete, we showed in \cite[Corollary 3.9]{BGVG} that $\Omega_T^\ast(F)$ is finite.
In the present article, we give an explicit bound in the case where the residue field $k$ is perfect, while dropping the assumption that $T$ is complete.
We show that $$\vert \Omega^*_T(F) \vert\leq \gen(F/K)+1\,$$ 
holds when $K$ is algebraically closed in $F$, and we give sufficient conditions for having strict inequality (\Cref{T:main-bound}). 

Our arguments are based on a reduction to the case where $F$ is the function field of an arithmetic surface $\mc{X}$ over $T$ (which we can assume to be relatively minimal), and on the analysis of the intersection theory for the special fiber $\mc{X}_s$ of this surface.
This relies on the observation in \cite[Proposition 3.7]{BGVG} that any valuation $v\in \Omega^*_T(F)$ is centred in the generic point of an irreducible component of $\mc{X}_s$.
The core of our method lies in the study of an associated intersection graph and the natural action of the absolute Galois group $\mc{G}$ of the residue field $k$ on this graph.
The vertices of this (bipartite) graph are given by the irreducible components of ${\mc X}_s\times_k k_\alg$ and by the intersection points of different components. It turns out that the valuations $v\in\Omega_T^\ast(F)$ that we are particularly looking for, namely those where the residue field extension $\kappa_v/k$ is a non-ruled function field of genus zero, are centred in components $C$ which have a special role under the natural action of $\mc{G}$: the orbit of any vertex adjacent to $C$ under the restricted action of the stabiliser subgroup in $\mc{G}$ has even length.
We refer to such vertices as \textit{pivot vertices}.
In Section~3, we study pivot vertices in a finite connected graph and give upper bounds for their number of pivot vertices and for the number of orbits of pivot vertices, in terms of the Betti number of the graph.

To apply these results, we rely on standard results from the intersection theory of the reduction of curves, as discussed in \cite[Chapters~9-10]{Liu}.
We obtain an upper bound on the Betti number of our intersection graph (which is related to, but different from the graphs discussed in \cite{Liu}) 
in terms of properties of components of the special fiber and their (arithmetic) genera. This is discussed in Section 4.

The  study of the Galois action on the intersection graph could be simplified in the special case where $K=\rr(\!(t)\!)$, because the absolute Galois group of $\rr$ has order $2$. 
In fact, our method is
crucially inspired by an argument which we learned from J.~Van Geel,  treating the case of an elliptic curve over $\rr(\!(t)\!)$ via the symmetry of its Kodaira type under complex conjugation.
\medskip

We would like to express our gratitude to Jan~Van Geel for having inspired this work. We further wish to thank Eva~Bayer-Fluckiger, Jean-Louis~Colliot-Th\'el\`ene, Parul~Gupta, Qing~Liu, Gonzalo~Manzano Flores, Michel Matignon, Adrian~Wadsworth, Marco Zaninelli and to the anonymous referee for various comments and suggestions.

\section{Preliminaries}

In this section we recall basic properties and tools from algebraic and arithmetic geometry, in particular in relation to projective curves and arithmetic surfaces.
We generally follow the terminology from \cite{Liu},  our standard reference.
\medskip

Let $K$ be a field. 
By a \textit{projective curve over $K$} we mean a one-dimensional $K$-scheme of finite type admitting a closed immersion into $\mathbb{P}^n_K$ for some $n\in \nat$.
Let $C$ be a projective curve over $K$.
We denote by $\gen(C/K)$ the \textit{arithmetic genus of $C$ over $K$} as defined in \cite[Definition 7.3.19]{Liu}, namely by
  $$\gen(C/K)=1- \dim_K H^{0}(C, \mc{O}_C)+\dim_K H^1(C, \mc{O}_C)\,,$$
where $\mathcal{O}_C$ denotes the structure sheaf of $C$.
By \cite[Corollary 5.2.27]{Liu} the arithmetic genus is invariant under base change, i.e.~we have $\gen(C\times_KL/L)=\gen(C/K)$ for any field extension $L/K$. 
In the sequel when speaking of the genus of a curve we refer to its arithmetic genus.

For a function field in one variable $F/K$, we denote by $\gen(F/K)$ the genus of $F/K$ as defined in \cite{Deu};
by its definition $\gen(F/K)$ remains the same when $K$ is replaced by the relative algebraic closure of $K$ in $F$.

\begin{prop}\label{P:Deu-gen}
Let $C$ be a geometrically integral projective curve over $K$ and let $F$ denote its function field.
Then $\gen(F/K)\leq\gen(C/K)$, and equality holds if and only if $C$ is regular.
\end{prop}
\begin{proof}
By the hypothesis on $C$, $K$ is relatively algebraically closed in $F$. We denote by $C'$ the normalization of $C$. Then $C'$ is regular and $\gen(F/K)=\gen(C'/K)$.
Furthermore, by \cite[Proposition 7.5.4]{Liu}, we have $\gen(C/K)\geq\gen(C'/K)$, and equality holds if and only if $C$ is regular.
\end{proof}

By a \textit{conic over $K$} we mean a projective curve over $K$ which is isomorphic to a curve in $\mathbb{P}^2_K$  given by a homogeneous polynomial of degree $2$ over~$K$. 
\begin{ex}\label{E:conics}
Let $f \in K[X,Y]$ be a square-free homogeneous polynomial of degree $2$. This gives rise to two conics in $\mathbb{P}_K^2$, namely:
\begin{enumerate}[$(i)$]
\item The regular conic in $\mathbb{P}_K^2$ given by $f(X,Y) = Z^2$.
\item The singular reduced conic in $\mathbb{P}^2_K$ given by $f(X,Y)=0$. 
\end{enumerate}
Clearly the two conics are non-isomorphic. Hence the two examples are of essentially different type.
Any reduced conic over $K$ is isomorphic to a conic given as in $(i)$ or in $(ii)$ by some square-free homogeneous quadratic polynomial $f\in K[X,Y]$.
\end{ex}

We will need to distinguish between different types of curves of genus zero.

\begin{prop}\label{P:genus0=conic}
Let $C$ be an integral projective curve over  $K$. If $C$ is a conic, then $\gen(C/K)=0$. Conversely, if  $\gen(C/K)=0$, then either $C$ is a regular conic over $K$, or $C$ is singular and birational over $K$ to $\mathbb{P}^1_{L}$ for some proper finite field extension $L/K$.
\end{prop}
\begin{proof}
If $C$ is a conic, then it follows by \cite[Tag 0BYD]{stacks-project} that $\gen(C/K)=0$. 

Assume now that $\gen(C/K)=0$. Set $M=\mc{O}_C(C)$.
Since $C$ is projective, connected and reduced, $M$ is a field, by \cite[Corollary 3.3.21]{Liu}. 
Hence, we can view $C$ as a curve over $M$. Since $H^0(C,\mc{O}_C)=\mc{O}_C(C)=M$,
we obtain that
\begin{align*}
0    =  \gen(C/K)  
  =   &\, 1-\dim_K H^0(C, \mc O_{C}) + \dim_K H^1(C, \mc O_{C}) \\
  =  & \, 1 - [M:K] \left( 1 -  \dim_M H^1(C, \mc O_{C})\right).
\end{align*}
This shows that $M=K$. 
The conclusion now follows by \cite[Tag 0C6N, Tag 0DWG]{stacks-project} if $C$ is regular and by  \cite[Tag 0DJB]{stacks-project} if $C$ is singular.
\end{proof}

A curve over $K$ is \textit{rational} if it is birational to $\mathbb{P}_K^1$. 
We call an integral curve over $K$ \textit{ruled} if it is birational as a $K$-scheme to $\mathbb{P}^1_L$ for some finite field extension $L/K$, or equivalently, if its function field over $K$ is ruled.

\begin{lem}\label{algebro-conics-are-pivot}
Let $C$ be an integral conic over $K$ which is singular or which  has a point of odd degree over $K$. Then $C$ is ruled.
More precisely, $C$ is birational over $K$ to $\mathbb{P}_L^1$ for a field extension $L/K$ with $[L:K]\leq 2$.
\end{lem}

\begin{proof}
Suppose first that $C$ is singular.  By \cref{E:conics} we may assume that $C$ is defined in $\mathbb{P}^2_K$ by $f(X,Y)=0$ for some homogeneous quadratic polynomial $f(X,Y) \in K[X,Y]$. As $C$ is integral, $f$ is irreducible in $K[X,Y]$.
Let $L$ denote the splitting field of the irreducible quadratic polynomial $f(X,1)$ over $K$. 
Then $[L:K]=2$ and the function field of $C$ over $K$ is isomorphic to a rational function field over $L$, whereby $C$ is birational to $\mathbb{P}_L^1$.

Assume now that $C$ is regular and contains a point of odd degree over $K$.
By \cite[Corollary 18.5]{EKM} this implies that $C$ has a $K$-rational point.
If $C$ is smooth over $K$, then we conclude that $C$ is isomorphic to $\mathbb{P}^1_K$. 
Suppose now that $C$ is not smooth over $K$. Then $\car(K) =2$ and, since $C$ is regular and has a $K$-rational point, we get that $C$ 
is given up to isomorphism by an equation $X^2 + a Y^2 =Z^2$ for some $a \in \mg K \setminus \sq K$. It follows that $C$ is isomorphic to $\mathbb{P}^1_{L}$ for $L=K(\sqrt{a})$.
\end{proof}

\begin{lem}\label{L:reduct-to-smooth}
Let $F/K$ be a function field in one variable such that $K$ is relatively separably closed in $F$.
Let ${F_\alg}$ be an algebraic closure of $F$.
There exists a finite purely inseparable extension $K'/K$ in ${F_\alg}$ such that 
the compositum $FK'$ in ${F_\alg}$ is the function field of a geometrically connected smooth projective curve over~$K'$.
\end{lem}
\begin{proof}
Let $K_{{\per}}$ be the maximal purely inseparable algebraic extension of $K$ inside ${F_\sep}$. Then $K_{{\per}}$ is perfect and $FK_{\per}/K_{\per}$ is a function field in one variable.
Since $K_{\per}/K$ is purely inseparable, so is $FK_{\per}/F$.
It follows that $K$ is relatively separably closed in $FK_{\per}$.
Hence $K_{\per}$ is relatively algebraically closed in $FK_{\per}$.

By \cite[Proposition~7.3.13]{Liu}, $FK_{\per}$ is the function field of a normal projective curve $C'$ over $K_{\per}$.
Then $C'$ is regular, and as $K_{\per}$ is perfect, $C'$ is smooth over $K_{\per}$, by \cite[Corollary 4.3.30]{Liu}.
Moreover, since $K_{\per}$ is relatively algebraically closed in $FK_{\per}$, we have that $C'$ is  geometrically integral, by  \cite[Corollary 3.2.14]{Liu}.

We choose a representation of $C'$ by equations inside projective $n$-space over $K_\per$ for some $n\in\nat$.
We further choose a finite extension $K'/K$ contained in $K_\per$ such that $FK'$ contains the coefficients of these equations as well as the generators of $FK_\per/K_\per$ determined by the variables.
Then $FK'$ is the function field of a projective curve $C$ over $K'$ defined by the same equations, and as $C'=C\times_{K'}K_\per$ it follows that $C$ is smooth over $K'$.
\end{proof}

Let $T$ be a discrete valuation ring, $K$ its field of fractions and $k$ its residue field.
 A \textit{fibered surface over $T$}  is a $2$-dimensional integral scheme $\mc{X}$ together with a flat projective morphism $\mc{X} \to \mathsf{Spec}(T)$. 
 
Let $\mc{X}$ be a normal fibered surface over $T$.
We call  $\mc{X}_K=\mc{X} \times_T K$ the \textit{generic fiber of $\mc{X}$ over $T$} and we call $\mc{X}_s=\mc{X} \times_T k$ the \textit{special fiber of $\mc{X}$ over $T$}. 
By \cite[Lemma~8.3.3]{Liu}, $\mc{X}_K$ is an integral normal projective curve over $K$ and $\mc{X}_s$ is a projective curve over $k$. The function field of $\mc{X}$ is equal to the function field of the generic fiber $\mc{X}_K$. If $\mc{X}$ is regular then we call $\mc{X}$ an \textit{arithmetic surface over $T$}.

An arithmetic surface $\mc{X}$ over $T$ is called \textit{relatively minimal} if any birational morphism $\mc{X} \to \mc{Y}$ of arithmetic surfaces over $T$ is an isomorphism, and it is called  \textit{minimal} if every birational map $\mc{Y} \dashrightarrow \mc{X}$ of arithmetic surfaces over $T$ extends to a morphism. 

\begin{lem}\label{L:model-unramif-ext}
Let $T'/T$ be an extension of discrete valuation rings which is a direct limit of \'etale extensions of $T$.
Let $\mc{X}$ be an arithmetic surface over~$T$. 
Then $\mc{X}'=\mc{X}\times_TT'$ is an arithmetic surface over~$T'$.
If the generic fiber $\mc{X}_K$ is integral and geometrically connected and $\mc{X}$ is relatively minimal over $T$, then either $\mc{X}'$ is relatively minimal over $T'$, or  $\mc{X}_s$ is integral and ruled with $\gen(\mc{X}_s/k)=0$.
\end{lem}
\begin{proof}
Let $k'$ denote the residue field of $T'$. 
Then the special fiber of $\mc X'$ is ${\mc X}_s\times_k k'$.
By \cite[Proposition 3.2.7]{Liu}, \cite[Proposition 4.3.3]{Liu} and \cite[Corollary 3.3.32]{Liu}, 
$\mc{X}'$ is a $2$-dimensional flat projective scheme over $T'$.
Since $\mc{X}$ is regular and $T'/T$ is a direct limit of \'etale extensions, it follows by \cite[Corollary 4.3.24]{Liu} that ${\mc X}'$ is regular.
Therefore $\mc{X}'$ is an arithmetic surface over $T'$.

Assume now that $\mc{X}_K$ is integral and geometrically connected.
By \cite[Corollary 8.3.6]{Liu}, we obtain that $\mc{X}_s$ is geometrically connected and $\gen(\mc{X}_s/k)=\gen(\mc{X}_K/K)$.  
Assume further that $\mc X$ is relatively minimal.

First consider the case where $\gen(\mc{X}_K/K)\geq 1$.
Then $\mc X$ is minimal over $T$, by \cite[Corollary 9.3.24]{Liu}.
Since $T'/T$ is a direct limit of \'etale extensions, $\mc{X}'$ is minimal over $T'$, by \cite[Proposition 9.3.28]{Liu}, so in particular relatively minimal.

Suppose now that $\gen(\mc{X}_K/K)=0$.
Since $\mc{X}_K$ is integral, it follows by
 \Cref{P:genus0=conic} that $\mc{X}_K$ is a regular conic over $K$.
Furthermore, $\mc{X}_s$ is integral, by \cite[Exercise 9.3.1]{Liu}.
By \cite[Corollary 8.3.6]{Liu}, we have $\gen(\mc{X}_s/k)=\gen(\mc{X}_K/K)=0$.  
If $\mc{X}_s$ is geometrically irreducible, then 
$\mc{X}'$ is relatively minimal.
Assume now that $\mc{X}_s$ is not geometrically irreducible.
Since $\mc{X}_s$ is geometrically connected, it follows that $\mc{X}_s$ is not regular.
Hence \Cref{P:genus0=conic} yields that $\mc{X}_s$ is ruled.
\end{proof}

%%%%%%%%%%%%%%%%%%%%%%%%%%%%%%%%
\section{Counting pivot vertices in a graph}\label{S:pivot} %%%
%%%%%%%%%%%%%%%%%%%%%%%%%%%%%%%%

We consider a finite undirected graph $\mc{B}$.
We denote by $\mc{V}$ the set of vertices and by $\mc{E}$ the set of edges of $\mc{B}$. Hence $\mc{V}$ is a finite set and we can represent $\mc{E}$ as a set of $2$-element subsets of $\mc{V}$.
The \textit{degree} of a vertex $v\in \mc{V}$, denoted by $\deg(v)$, is defined by $\deg(v)=|\{\{v,v'\}\in\mc{E}\mid v'\in \mc{V}\}|$.

We consider an action of a group $G$ on $\mc{B}$  given by 
graph automorphisms. In other words, $G$ acts on the set $\mc{V}$ in such way that, for every $g\in G$ and every $v,v'\in \mc{V}$ with $\{v,v'\}\in\mc{E}$ we have $\{gv,gv'\}\in\mc{E}$.
We call $v\in \mc{V}$ a \textit{$G$-pivot vertex} (of $\mc{B}$) if,
for every $w\in \mc{V}$ with $\{v,w\}\in\mc{E}$, the orbit of $w$ under the stabiliser of $v$ in $G$ has even cardinality, i.e. $|\{gw \mid g\in G, gv=v\}|$ is even.

We simply call $v\in\mc V$ a \textit{pivot vertex of $\mc B$} if $v$ is a $G$-pivot vertex for $G=\mathsf{Aut}(\mc{B})$, the automorphism group of the graph $\mc{B}$.
For any group $G$ acting on $\mc{B}$, a $G$-pivot vertex of $\mc{B}$ is in particular a pivot vertex of $\mc{B}$. 

\begin{prop}\label{pivots in trees}
If $\mc{B}$ is a tree, then $\mc{B}$ has at most one pivot vertex.
\end{prop}
\begin{proof}
If $|\mc{V}|\leq 1$ then the statement is trivial.
Assume now that  $\mc{B}$ is a tree with $|\mc{V}| > 1$. 
Then $\mc{B}$ has at least one vertex of degree $1$.
By removing all vertices of degree $1$ from $\mc{B}$, we obtain a  subtree $\mc{B}'$ of $\mc B$ with strictly fewer vertices than $\mc{B}$, while all pivot vertices of $\mc{B}$ are also pivot vertices of $\mc{B}'$.
Hence the statement follows by induction on $|\mc{V}|$.
\end{proof}

We denote by $\beta(\mc{B})$ the \textit{Betti number of $\mc{B}$}, which is given by the formula
\begin{eqnarray*}\beta(\mc{B}) & = & |\mc{E}|-|\mc{V}|+1\,.\end{eqnarray*}
Note that for a nontrivial connected graph $\mc{B}$ we have $\beta(\mc{B})\geq 0$, and equality holds if and only if $\mc{B}$ is a tree.

\begin{lem}\label{L:G-fixpoint-pivot-count}
Let $\mc{W}$ denote the set of $G$-pivot vertices in $\mc{B}$.
Assume that $\mc{B}$ is connected and contains a vertex $v_0$ which is fixed by $G$.
Then $|\mc{W}\setminus\{v_0\}|\leq \beta(\mc{B})$.  
\end{lem}
\begin{proof}
For $v\in\mc{V}$ let $d(v)$ denote the distance between $v$ and $v_0$, that is,
the smallest $r\in\nat$ such that there exist $v_1,\ldots,v_r\in\mc{V}$ with $\{v_{i-1},v_i\}\in\mc{E}$ for $1\leq i\leq r$ and $v_r=v$.
Note that $d(gv)=d(v)$ for all $v\in\mc{V}$ and $g\in G$.

We prove the statement by induction on $|\mc{E}|$.
If $|\mc{E}|=0$, then $\mc{V}=\{v_0\}$ and the statement holds trivially.
Assume now that $|\mc{E}|\geq 1$.
Set $m=\max\{d(v)\mid v\in\mc{V}\}$, $\mc{M}=\{v\in\mc{V}\mid d(v)=m\}$ and $\mc{E}^\ast=\mc{E}\cap\{\{w,w'\}\mid w,w'\in\mc{M}\}$.

We first consider the case where $\mc{E}^\ast\neq \emptyset$.
We set $\mc{E}'=\mc{E}\setminus\mc{E}^\ast$ and consider the graph $\mc{B}'=(\mc{V},\mc{E}')$.
Note that $\mc{B}'$ is connected, $\beta(\mc{B}')<\beta(\mc{B})$ and the $G$-action on $\mc{B}$ restricts to a $G$ action on $\mc{B}'$.
Furthermore, every $G$-pivot vertex of $\mc{B}$ is also a $G$-pivot vertex of $\mc{B}'$.
Since $|\mc{E}'|<|\mc{E}|$, we therefore obtain by the induction hypothesis that 
$|\mc{W}\setminus\{v_0\}|\leq \beta(\mc{B}')< \beta(\mc{B})$.

We now consider the case where $\mc{E}^\ast=\emptyset$.
Set $\mc{V}'=\mc{V}\setminus \mc{M}$. We consider the graph $\mc{B}'=(\mc{V}',\mc{E}')$ where $\mc{E}'=\mc{E}\cap\{\{v,v'\}\mid v,v'\in\mc{V}'\}$ (i.e.~the full subgraph of $\mc{B}$ spanned by $\mc{V}'$).
Note that $\mc{B}'$ is connected,  $\beta(\mc{B}')\leq \beta(\mc{B})$, the $G$-action on $\mc{B}$ restricts to a $G$-action on $\mc{B}'$, and the elements of $\mc{W}\cap\mc{V}'$ are also $G$-pivot vertices in $\mc{B}'$. 
Since $\mc{M}\neq \emptyset$ we have that $|\mc{E}'|<|\mc{E}|$.
Hence, the induction hypothesis yields that  $|(\mc{W}\cap\mc{V}')\setminus\{v_0\}|\leq \beta(\mc{B}')$. 
Since $\mc{E}^\ast=\emptyset$, we have $|\mc{E}\setminus\mc{E}'| = \sum_{w \in \mc{M}} \deg(w)$ and therefore
$\beta(\mc{B})-\beta(\mc{B}')=\sum_{w\in\mc{M}}(\deg(w)-1)$.
Since every $G$-pivot vertex of $\mc{B}$ has degree at least $2$, we obtain that 
 $\beta(\mc{B})-\beta(\mc{B}')\geq |\mc{W}\cap\mc{M}|$.
Hence we conclude that 
 $|\mc{W}\setminus\{v_0\}|= |(\mc{W}\cap \mc{V}')\setminus\{v_0\}|+|\mc{W}\cap\mc{M}|\leq \beta(\mc{B})$.
\end{proof}

\begin{lem}\label{L:pivot-betti-bound}
Assume that $\mc{B}$ is connected and let $\mc{W}$ be the set of $G$-pivot vertices in $\mc{B}$.
Let $v\in \mc V$. 
Then
$|\mc{W}\setminus Gv| \,\leq \, \beta(\mc{B})+|Gv|-1$.
\end{lem}
\begin{proof}
We apply \Cref{L:G-fixpoint-pivot-count} to a graph obtained from $\mc{B}$ by adding one vertex and connecting it with all vertices in $Gv$.
Let $v_0$ denote an extra vertex, not contained in $\mc{V}$, and set $\mc{V}_0={\mc V}\cup\{v_0\}$ and 
$\mc{E}_0=\mc{E}\cup\{\{v_0,gv\}\mid g\in G\}$.
We obtain that $\mc{B}_0=(\mc{E}_0,\mc{V}_0)$ is a connected graph with $\beta(\mc{B}_0)=\beta(\mc{B})+|Gv|-1$.
By letting $gv_0=v_0$ for all $g\in G$ we extend the $G$-action on $\mc{B}$ to a $G$-action on $\mc{B}_0$.
Note that none of the vertices in $Gv$ is a $G$-pivot vertex in $\mc{B}_0$, and $v_0$ is a $G$-pivot vertex of $\mc{B}_0$ if and only if $|Gv|$ is even.
Hence the set of $G$-pivot vertices of $\mc{B}_0$ is either $(\mc{W}\setminus Gv)\cup\{v_0\}$ or $\mc{W}\setminus Gv$.
We conclude by \Cref{L:G-fixpoint-pivot-count} that $|\mc{W}\setminus Gv|\leq \beta(\mc{B}_0)=\beta(\mc{B})+|Gv|-1$.
\end{proof}

\begin{thm}\label{T:pivot-betti-bound}
Assume that $\mc{B}$ is connected.
Let $\mc W$ be the set of $G$-pivot vertices in $\mc{B}$ and assume that $\mc{W}\neq \emptyset$.
Let $d=\min\{|Gv|\mid v\in\mc{W}\}$.
Then $$|\mc{W}|\leq \beta(\mc{B})+2d-1\,.$$
Furthermore, the number of $G$-orbits in $\mc{W}$ is bounded by $\frac{1}d(\beta(\mc{B})-1)+2$.
\end{thm}
\begin{proof}
Let $v_0\in\mc{W}$ be such that $|Gv_0|=d$.
It follows by \Cref{L:pivot-betti-bound} that $$|\mc{W}|=|\mc{W}\setminus Gv_0|+d\leq \beta(\mc{B})+2d-1\,.$$
Let $N$ denote the number of $G$-orbits in $\mc{W}$.
Then $Nd\leq |\mc{W}|\leq \beta(\mc{B})+2d-1$ and
therefore $N\leq \frac{1}d(\beta(\mc{B})-1)+2$.
\end{proof}

\begin{cor}\label{C:pivot-orbit-bound-opt}
Assume that $\mc{B}$ is connected.
Let $N$ be the number of $G$-orbits of $G$-pivot vertices in $\mc{B}$.
Then $N\leq \beta(\mc{B})+1$.
Moreover, if $N=\beta(\mc{B})+1$, then either the $G$-pivot vertices in $\mc{B}$ coincide with the $G$-fixed vertices in $\mc{B}$, or
$\beta(\mc{B})=1$ and $\mc{B}$ has no $G$-fixed vertices.
\end{cor}
\begin{proof}
Let $\mc W$ be the set of $G$-pivot vertices in $\mc{B}$.
If $\mc{B}$ has a $G$-fixed vertex $v$ which is not a $G$-pivot vertex, then 
it follows by \Cref{L:G-fixpoint-pivot-count} that $|\mc{W}|\leq \beta(\mc{B})$, whereby $N\leq \beta(\mc{B})$.

We now assume that all $G$-fixed vertices belong to~$\mc W$.
We further assume that $\mc{W}\neq \emptyset$, because otherwise there is nothing to show.
Set $d=\min\{|Gv|\mid v\in\mc{W}\}$.
By \Cref{T:pivot-betti-bound}, we have that $N\leq \frac{1}d(\beta(\mc{B})-1)+2$.
If $\beta(\mc{B})\geq 2$ and $d\geq 2$, then $\frac{1}d(\beta(\mc{B})-1)+2<\beta(\mc{B})+1$, whereby $N\leq \beta(\mc{B})$.
If $\beta(\mc{B})=1$ and $d\geq 2$, then $\mc{B}$ has no $G$-fixed vertex and $N\leq 2=\beta(\mc{B})+1$.
If $\beta(\mc{B})=0$, then $\mc{B}$ is a tree, and as $\mc{W}\neq\emptyset$, it follows by \Cref{pivots in trees} that $N=|\mc{W}|=1=\beta(\mc{B})+1$, so that $\mc{W}$ consists of the unique $G$-fixed vertex of $\mc{B}$.

Assume finally that $d= 1$. Then \Cref{T:pivot-betti-bound} yields that $N\leq |\mc{W}|\leq \beta(\mc{B})+1$. 
Assume now in addition that $N=\beta(\mc{B})+1$. 
Then $N=|\mc{W}|$, which implies that the elements of $\mc{W}$ are fixed by $G$. Since all $G$-fixed vertices belong to $\mc{W}$, we conclude that $\mc{W}$ consists of the $G$-fixed vertices in $\mc{B}$.
\end{proof}

\begin{ex}\label{EX:dihedral}
Let $n\geq 3$ and let $\mc{B}$ be the circle graph with $n$ vertices.
Note that $\beta(\mc{B})=1$.
The automorphism group $G$ of $\mc{B}$ is the dihedral group of order $2n$.
Every vertex of $\mc{B}$ is a $G$-pivot vertex, and they all lie in one $G$-orbit. 

Assume now that $n$ is even. 
Then $G$ has a unique noncyclic subgroup $G'$ of index $2$, the dihedral group of order $n$. Under the action of $G'$ on $\mc{B}$, the $n$ vertices lie in two disjoint $G'$-orbits.
Moreover, if $n\equiv 2\bmod 4$, then all vertices of $\mc{B}$ are also $G'$-pivot vertices, hence in this case the number of $G'$-orbits of $G'$-pivot vertices is equal to $2$.
\end{ex}

The example shows that there is no bound on the total number of pivot vertices in a connected graph $\mc{B}$ with $\beta(\mc{B})=1$. 
This case is nevertheless an exception.

\begin{thm}\label{T:total-pivot-bound}
Let $\mc{B}$ be a connected graph with $\beta(\mc{B})\geq 2$.
Then the total number of pivot vertices in $\mc{B}$ is bounded by $3(\beta(\mc{B})-1)$.
\end{thm}
\begin{proof}
To prove the statement it is convenient to reformulate it as follows.
We fix a group $G$ acting on the graph $\mc{B}$ and claim that the number of $G$-pivot vertices in $\mc{B}$ is bounded by $3(\beta(\mc{B})-1)$.
Taking for $G$ the automorphism group of the graph $\mc{B}$ then yields the theorem.

Note that neither the set of $G$-pivot vertices, nor the Betti number change if we remove from the graph $\mc{B}$ all vertices of degree $1$, and the graph remains connected under this reduction.
We may therefore reduce to the case where $\deg(v)\geq 2$ holds for all $v\in\mc{V}$.
Since $\beta(\mc{B})-1=|\mc{E}|-|\mc{V}|$ we have 
\begin{eqnarray}\label{E:betti-deg}
2(\beta(\mc{B})-1) & = & \sum_{v\in\mc{V}} (\deg(v)-2) \,.
\end{eqnarray}
Let $\mc{W}$ be the set of $G$-pivot vertices of $\mc{B}$ and $\mc{V}_+=\{v\in\mc{V}\mid \deg(v)\geq 3\}$.
As $\beta(\mc{B})\geq 2$, we have $\mc{V}_+\neq\emptyset$, so we fix some $v\in\mc{V}_+$ of minimal degree.
Since $Gv\subseteq \mc{V}_+$ we have $|Gv|\leq |\mc{V}_+|$.

Assume first that $\deg(v)=3$. 
Then $v$ is not a pivot vertex, so $Gv\cap \mc{W}=\emptyset$.
Since $\deg(v)=2$ for all $v\in\mc{V}\setminus\mc{V}_+$, we obtain by \eqref{E:betti-deg} that
$|\mc{V}_+|\leq 2(\beta(\mc{B})-1)$.
Using \cref{L:pivot-betti-bound}, we conclude that 
$|\mc{W}|=|\mc{W}\setminus Gv|\leq \beta(\mc{B})+|Gv|-1\leq 3(\beta(\mc{B})-1)$.

Assume now that $\deg(v)\geq 4$.
Since we have chosen $v$ of minimal degree in $\mc{V}_+$, we can now conclude from \eqref{E:betti-deg} that $|\mc{V}_+|\leq \beta(\mc{B})-1$.
Applying \cref{L:pivot-betti-bound} again, we obtain that
$|\mc{W}|\leq |\mc{W}\setminus Gv|+|Gv|\leq \beta(\mc{B})+2|Gv|-1\leq 3(\beta(\mc{B})-1)$.
\end{proof}

\begin{ex}
Let $n$ be a positive integer. Let $\mc{B}$ be the graph with $3n$ vertices 
$a_1,b_1,c_1,\dots,a_n,b_n,c_n$ such that 
$a_i$ is connected to the four vertices $b_{i-1},c_{i-1},b_i,c_i$ for $1\leq i\leq n$, where $b_0=b_n$ and $c_0=c_n$.
(Hence the graph $\mc{B}$ resembles a necklace composed of $n$ rings.)
It is easy to see that $\beta(\mc{B})=n+1$ and that all vertices of $\mc{B}$ are pivot vertices. Hence in this case the upper bound $3(\beta(\mc{B})-1)$ from \Cref{T:total-pivot-bound} for the number of pivot vertices in $\mc{B}$ is attained.
\end{ex}

%%%%%%%%%%%%%%%%%%%%%%%%%%%%%%%%%%%%%%%%%%%%%%
\section{Galois action on the intersection graph of a curve}\label{intersection graphs}   %%%
%%%%%%%%%%%%%%%%%%%%%%%%%%%%%%%%%%%%%%%%%%%%%%

Let $k$ be a perfect field. Let ${k_{\alg}}$ denote an algebraic closure of $k$ and $\mc{G}$ the absolute Galois group of $k$. 
Let $X$ be an algebraic variety over $k$. When we apply scheme theoretic terms to closed subsets of $X$ we implicitly view them as closed subschemes with their induced reduced structure.
We set  $\ovl{X}=X\times_k{k_{\alg}}$, which is the base change of $X$ to ${k_{\alg}}$.  
The projection $\pi: \ovl{X}\rightarrow X$ is a surjective closed morphism of $k$-schemes.
The group $\mc G$ acts naturally on ${k_{\alg}}$. Letting $\mc{G}$ act trivially on  $X$, we obtain a natural action of $\mc G$ on $\ovl{X}$.
With respect to this action, $\pi$ is $\mc{G}$-equivariant. 
  For every extension $\ell/k$ contained in  $k_\alg$, we set $X_\ell = X \times_k \ell$
and we denote by $\pi_\ell: \overline{X} \to X_\ell$ the corresponding projection morphism.

\begin{prop}\label{P:G-act-closed-subspace}
Assume that $X$ is integral. 
Let $C$ be an irreducible component of $\ovl{X}$. 
Let $\ell\subseteq k_\alg$ denote the fixed field of $\{\rho\in\mc{G}\mid \rho(C)=C\}$ and $C_0=\pi_\ell(C)$. 
Then $C=C_0 \times_\ell k_\alg$ and $k(X)\simeq \ell(C_0)$. 
\end{prop}
\begin{proof}
Since $\pi_\ell$ is a closed morphism, it is clear that $C_0$ is an irreducible component of $X_\ell$.   
By \cite[Theorem 4.2.4]{MO15}, the group $\mc{H}=\{\rho\in\mc{G}\mid \rho(C)=C\}$ acts transitively on the set of irreducible components of $C_0\times_\ell k_\alg$.
Since $C$ is an irreducible component of $C_0\times_\ell k_\alg$ and is stabilized by $\mc{H}$, we conclude that $C = C_0 \times_\ell k_\alg$.

Let ${\ell'}$ denote the relative algebraic closure of $k$ in $k(X)$ and let $\mc{E}$ be the set of $k$-embeddings $\ell'\to k_\alg$. 
Then $|\mc{E}|=[{\ell'}:k]<\infty$ and, by \cite[Theorem 4.2.4]{MO15}, there is a $\mc{G}$-invariant bijection between $\mc{E}$ and the set of irreducible components of $\ovl{X}$. 
Hence $\mc{H}$ is also the stabilizer of some $\rho\in\mc{E}$.
Then $\rho({\ell'}) \subseteq \ell$, and since $[\ell:k]=[\mc{G}:\mc{H}]=|\mc{E}|=[{\ell'}:k]$, we conclude that $\rho({\ell'}) = \ell$.

In order to show that $k(X) \simeq \ell(C_0)$, we may assume that $X$ is regular and affine. In particular, $\mc{O}_X(X)$ is normal and thus ${\ell'} \subseteq \mc{O}_X(X)$.
Hence $X$ is a variety over ${\ell'}$.
Since $k$ is perfect, we have
 $${\ell'} \otimes_k \ell \simeq \prod_{i=1}^s \ell_i$$ 
 for some $s\leq [\ell:k]$ and some finite field extensions $\ell_i/\ell$ for $1 \leq i \leq s$.
 Hence 
 $$X_\ell\,=\,X \times_k \ell \,\simeq\, (X \times_{{\ell'}} {\ell'}) \times_k \ell \,\simeq\, X \times_{{\ell'}} \Spec({\ell'} \otimes_k \ell) \,\simeq\, \bigsqcup_{i=1}^s X \times_{{\ell'}} \ell_i.$$
For $1\leq i\leq s$, the component $X \times_{{\ell'}} \ell_i$ is irreducible, and it is geometrically irreducible as an $\ell$-variety if and only if $\ell_i=\ell$. Hence, those components of $X_\ell$ which are geometrically irreducible as $\ell$-varieties are isomorphic as  $k$-varieties to $X$.
As $C_0$ is geometrically irreducible over $\ell$, we conclude that $k(X) \simeq \ell(C_0)$. 
 \end{proof}

Let  $X$ be a projective curve over $k$.
Then $\ovl{X}$ is a projective curve over ${k_{\alg}}$, by \cite[Proposition 3.2.7 and Corollary 3.3.32]{Liu}.
Let $\mc{C}$ be the set of irreducible components of $\ovl{X}$ and let $\mc{P}$ be the set of those points on $\ovl{X}$ where at least two distinct irreducible components of $\ovl{X}$ intersect.
Let $$\mc{E}=\{(Y,P)\in\mc{C}\times\mc{P}\mid P\mbox{ lies on } Y\}\,.$$
We define the bipartite graph $$\mc{B}_X=(\mc{C},\mc{P},\mc{E})\,,$$
 hence  taking the elements of $\mc{C}$ and of $\mc{P}$ as the vertices of two different colors, say \textit{cyan} and \textit{pink}, and by letting $Y\in\mc{C}$ and $P\in\mc{P}$ be linked by an edge if and only if the point $P$ lies on the component $Y$.
We call $\mc B_{X}$ the \textit{bipartite intersection graph of $X$}.
We also view $\mc B_{X}$ as a uncolored graph with vertex set $\mc{C}\cup\mc{P}$ and edge set $\{\{C,V\}\mid (C,V)\in\mc E\}$ to apply the setup of  the previous section.
 
\begin{prop}
The graph $\mc{B}_X$ is connected if and only if the curve $X$ is geometrically connected.
\end{prop}
\begin{proof}
Since the curve $X$ is projective over $k$, so is the curve $\ovl{X}$ over $k_\alg$.
Two points $x$ and $x'$ of $\ovl{X}$ belong to the same connected component of $\ovl{X}$ if and only if there exists $r\in\nat$ and a sequence of irreducible components $Y_0,\dots,Y_r$ of $\ovl{X}$ such that $x\in Y_0$, $x'\in Y_r$ and such that $Y_{i-1}\cap Y_i\neq\emptyset$ for $1\leq i\leq r$.
Since any intersection point of two different irreducible components of $\ovl{X}$ belongs to $\mc{P}$, it follows that the connected components of $\ovl{X}$ correspond to the connected components of the graph $\mc{B}_X$.
Hence, the graph $\mc{B}_X$ is connected if and only if $\ovl{X}$ is connected as a curve, which is if and only if $X$ is geometrically connected.
\end{proof}

We will now take into account the natural action of the absolute Galois group $\mc{G}$ on the $k$-scheme $\ovl{X}$.
For any $\s\in\mc{G}$  we have $\s(\mc{C})=\mc{C}$, $\s(\mc{P})=\mc{P}$,  
and furthermore $(\s(C),\s(P))\in\mc{E}$ for any $(C,P)\in\mc{E}$.
Hence the $\mc{G}$-action on $\ovl{X}$ induces a $\mc{G}$-action on the graph $\mc{B}_X$ which preserves the colors of the vertices.

\begin{prop}\label{P:conic-pivot}
Let $C$ be an irreducible component of $\ovl{X}$ such that  $C\simeq \mathbb{P}^1_{{k_{\alg}}}$.
If $k(\pi(C))/k$ is not ruled, then $C$ is a $\mc G$-pivot vertex of $\mc{B}_X$.
\end{prop}
\begin{proof}
Let $\mc{H}$ be the stabilizer of $C$ under the $\mc{G}$-action  on $\mc{B}_X$, 
 let $\ell$ be the fixed field of $\mc{H}$  in $k_\alg$ and $C_0 = \pi_\ell(C)$. 
 Then  $C\simeq C_0\times_{\ell}{k_{\alg}}$ and $k(\pi(C))\simeq \ell(C_0)$,  by \Cref{P:G-act-closed-subspace}. 
Since ${C}\simeq\mathbb{P}_{{k_{\alg}}}^1$, we have $\gen(C_0/\ell)=\gen(C/k_\alg)=0$.
It follows  by \Cref{P:genus0=conic} that $C_0$ is a regular conic over $\ell$.

Suppose that $k(\pi(C))/k$ is not ruled. As $k(\pi(C))\simeq \ell(C_0)$, it follows that $C_0$ is not rational over $\ell$.
By \cref{algebro-conics-are-pivot},  all closed points of $C_0$ have even degree over $\ell$.
Let $P$ be some vertex adjacent to $C$. Hence $P$ is an intersection point of $C$ with some other irreducible component of $\ovl{X}$.
Denote $Q=\pi_{\ell}(P)$. 
From \cite[Theorem 4.2.3]{MO15}  we conclude that the  $\mc{H}$-orbit of $P$ has length $[\ell(Q):\ell]$, which is even. This shows that $C$ is a  $\mc{G}$-pivot vertex of $\mc{B}_X$.
\end{proof}

%%%%%%%%%%%%%
\section{The bound} %%%
%%%%%%%%%%%%%

Let $T$ be a discrete valuation ring, $K$ its fraction field and $k$ its residue field.
Let $\mc{X}$ denote an arithmetic surface over $T$.
We write $\mc{B}_{\mc X}$ (instead of $\mc{B}_{\mc{X}_s}$) to denote the bipartite intersection graph  of its special fiber $\mc{X}_s$. 
For two divisors $D$ and $E$ on $\mc{X}$ contained in $\mc{X}_s$, we denote by $D\cdot E$ the intersection number of $D$ and $E$ as defined in  \cite[Definition 9.1.15]{Liu}.

\begin{lem}\label{L:special-betti-genus-bound}
Assume that $k$ is algebraically closed. 
Let $\mc{X}$ be a relatively minimal arithmetic surface over $T$ whose special fiber $\mc{X}_s$ is connected.
Let $n,d_1,\dots,d_n$ be the positive integers and $\Gamma_1,\ldots,\Gamma_n$ the distinct irreducible components of $\mc{X}_s$ such that $\mc{X}_s=\sum_{i=1}^n d_i\Gamma_i$ as a divisor on $\mc{X}$.
Let $I=\{1\leq i\leq n\mid \Gamma_i^2=-1\}$,
$$\epsilon=\sum_{i\in I}(\gen(\Gamma_i/k)-\half)(d_i-1)\quad\mbox{  and }\quad \hat{\beta}_{\mc{X}}\,=\,1\,-\,n\,+\!\!\sum_{1\leq i<j\leq n}\!\! \Gamma_i\cdot\Gamma_j\,.$$ 
Then $\epsilon\geq 0$ and
$$\beta(\mc{B}_{\mc{X}})\quad\leq\quad\hat{\beta}_{\mc{X}}\quad\leq\quad \gen(\mc{X}_s/k)\,-\,\sum_{i=1}^n \gen(\Gamma_i/k)\,-\,\epsilon\,.$$
Moreover, we have $\beta(\mc{B}_{\mc{X}})=\hat{\beta}_{\mc{X}}$ if and only if $\deg(P)=2$ for all $P\in\mc{P}$ and $\Gamma_i\cdot \Gamma_j\leq 1$ for $1\leq i<j\leq n$.
\end{lem}
\begin{proof}
Since $\mc{B}_{\mc{X}}$ is a bipartite graph with disjoint sets of vertices $\mc{P}$ and $\mc{C}$,
 we have
$$\beta(\mc{B}_{\mc{X}})=\Big(\sum_{P\in\mc{P}} \deg(P)\Big)-|\mc{P}|-|\mc{C}|+1\,\,.$$
The elements of $\mc{C}$ are given by the irreducible components 
$\Gamma_1,\dots,\Gamma_n$ of $\mc{X}_s$.
Hence $|\mc{C}|=n$ and 
$$\beta(\mc{B}_{\mc{X}})=1-n+\sum_{P\in\mc{P}} (\deg(P)-1)\,.$$
The total number of triples $(P,i,j)$ where $1\leq i<j\leq n$ and $P\in\mc{P}$ lies on both components $\Gamma_i$ and $\Gamma_j$ is given by $\sum_{P\in\mc{P}} \binom{\deg(P)}{2}$.
Since $\deg(P)\geq 2$ for all $P\in\mc{P}$, we obtain that
\begin{eqnarray}\label{IE:intersect}
\sum_{P\in\mc{P}} (\deg(P)-1) \quad\leq \quad  \sum_{P\in\mc{P}} \binom{\deg(P)}{2} & \leq & 
\sum_{1\leq i<j\leq n}\Gamma_i\cdot\Gamma_j \,.
\end{eqnarray}
This shows the first inequality in the statement.
Furthermore, the two inequalities in \eqref{IE:intersect} are both equalities if and only if $\deg(P)=2$ for all $P\in\mc{P}$ and $\Gamma_i\cdot \Gamma_j\leq 1$ for all $1\leq i<j\leq n$.

By Castelnuovo's criterion \cite[Theorem 9.3.8]{Liu}, since $\mc{X}$ is relatively minimal we have $\gen(\Gamma_i/k)\geq 1$ for every $i\in I$. This yields that $\epsilon\geq 0$.
From the adjunction formula \cite[Theorem~9.1.37]{Liu} applied to $\mc{X}_s$ and to $\Gamma_1,\dots,\Gamma_n$ we obtain that 
\begin{eqnarray*}
 2\gen(\mc{X}_s/k)-2-\sum_{i=1}^nd_i\cdot (2\gen(\Gamma_i/k)-2)  & = & \mc{X}_s\cdot\mc{X}_s-\sum_{i=1}^n d_i\Gamma_i^2\,.
\end{eqnarray*}
By \cite[Proposition~9.1.21]{Liu} we have $\mc{X}_s\cdot\mc{X}_s=0$. 
Hence 
\begin{eqnarray*}
\gen(\mc{X}_s/k)-\sum_{i=1}^nd_i \gen(\Gamma_i/k) -1+n & = & -\half\sum_{i=1}^nd_i\Gamma_i^2 + \sum_{i=1}^n(1-d_i)\,.
\end{eqnarray*}
For $1\leq i\leq n$, we have $\Gamma_i^2=-1$ if $i\in I$ and $\Gamma_i^2\leq -2$ otherwise.
Therefore
\begin{eqnarray*}
\label{IE2}
2\sum_{i=1}^n(1-d_i)-\sum_{i\in I}(1-d_i) & \geq & \sum_{i=1}^n (1-d_i)(-\Gamma_i^2)\,\,=\,\,\sum_{i=1}^nd_i\Gamma_i^2 -\sum_{i=1}^n\Gamma_i^2\,.
\end{eqnarray*}
It follows that
\begin{eqnarray*}
\gen(\mc{X}_s/k)-\sum_{i=1}^nd_i \gen(\Gamma_i/k) -1+n & \geq & - \half  \sum_{i=1}^n \Gamma_i^2  +\half   \sum_{i\in I}(1-d_i)\,.
\end{eqnarray*}
Since 
$\half  \sum_{i\in I}(1-d_i) = \epsilon - \sum_{i\in I}(d_i-1) \gen(\Gamma_i/k)\,,$
we obtain that
\begin{eqnarray*}
\gen(\mc{X}_s/k)-\sum_{i=1}^n \gen(\Gamma_i/k) -1+n & \geq & - \half \sum_{i=1}^n \Gamma_i^2  +\epsilon\,.
\end{eqnarray*}
By \cite[Theorem 9.1.23]{Liu},
we have
\begin{equation*}
\sum_{i=1}^n\Gamma_i^2+2\!\!\!\!\sum_{1\leq i<j\leq n}\!\!\!\!\Gamma_i\cdot\Gamma_j = \left(\sum_{i=1}^n\Gamma_i\right)^{\!2} \,\, \leq \,\, 0\,.\label{IE3}
\end{equation*}
Hence we obtain the second inequality in the statement.
\end{proof}

\begin{rem} 
The quantity $\hat{\beta}_{\mc{X}}$ appearing in \Cref{L:special-betti-genus-bound} is the Betti number of the dual graph of the reduction type of $\mc{X}$ as defined in \cite[Definition 10.1.48]{Liu}.
Under the assumption that $\mc{X}_s$ is reduced, 
 it is shown in \cite[Proposition 10.1.51]{Liu} that
$$\hat{\beta}_{\mc{X}}\quad=\quad \gen(\mc{X}_s/k)\,-\,\sum_{i=1}^n \gen(\Gamma_i/k)\,,$$
which permits a short-cut in the proof of \Cref{L:special-betti-genus-bound} in this case. 
 \end{rem}

\begin{thm}\label{T:main-bound}
Assume that $k$ is perfect. Let $F/K$ be a function field in one variable such that $K$ is relatively algebraically closed in $F$.
For $v\in\Omega_T(F)$ let $\ell_v$ be the relative algebraic closure of $k$ in $\kappa_v$.
Set $\Omega_T^{\circ}(F)=\{v\in\Omega_T^\ast(F)\mid \gen(\kappa_v/k)=0\}$.
Then 
\begin{eqnarray}
|\Omega_T^\ast(F)|\,\,\,\leq &  |\Omega_T^{\circ}(F)|\,\,+\!\!\!\!\sum\limits_{v\in\Omega_T^{\ast}(F)}[\ell_v:k]\cdot \gen(\kappa_v/k) \,\,\leq & 1+ \gen(F/K).\quad\label{IE:statement}
\end{eqnarray}
Moreover, if $|\Omega_T^\ast(F)| = \gen(F/K) + 1 $, then $\Omega_T^\ast(F)=\Omega_T^\circ(F)$ and 
$T$ is not totally ramified in the residue field extension $\kappa_v/K$ for any $v\in\Omega(F)$ with $K\subseteq\mc{O}_v$.
\end{thm}
\begin{proof}
As $\gen(\kappa_v/k)\geq 1$ for any $v\in\Omega_T^\ast(F)\setminus\Omega_T^{\circ}(F)$,
the first inequality in \eqref{IE:statement} is obvious.
To prove the other parts, we first reduce to the case where $F$ is the function field of a geometrically integral smooth projective curve over $K$.

By \Cref{L:reduct-to-smooth},
there exists a finite purely inseparable extension $K'/K$ and a function field in one variable $F'/K'$ containing $F$ with $F'=K'F$ and such that $F'$ is the function field of  a geometrically connected smooth projective curve $C$ over $K'$. Then $C$ is geometrically integral and $\gen(F'/K')=\gen(C/K')$, by \Cref{P:Deu-gen}.
We shall see that it is sufficient to prove the statement for the situation where $K'=K$ (which is certainly already the case when $\car(K)=0$).

Since $K'/K$ is a finite purely inseparable extension, there exists a unique discrete valuation ring $T'$ of $K'$ with $T=T'\cap K$. 
Since the residue field $k$ of $T$ is perfect and the extension $K'/K$ is purely inseparable, $k$ is also the residue field of $T'$.
By \cite[Chapter IV, \S 38, Theorem 3]{Deu}, we have $\gen(F'/K')\leq \gen(F/K)$.
 
Since $F'/F$ is a finite purely inseparable extension, for every $v\in \Omega_T(F)$ 
there exists a unique valuation $v'\in\Omega_{T'}(F')$ with $\mc{O}_{v'}\cap F=\mc{O}_v$.
We now consider $v\in\Omega_T(F)$ and look at the residue fields $\kappa_v$ and $\kappa_{v'}$.
Since $k$ is perfect and $\kappa_{v'}/\kappa_v$ is a finite purely inseparable extension, it follows by \cite[Proposition 3.10.2~(c)]{Stichtenoth} that  $\gen(\kappa_{v'}/k)=\gen(\kappa_v/k)$ and that $\kappa_{v'}/k$ is ruled if and only if $\kappa_v/k$ is ruled.

As $F'=K'F$ we have $[F':F]\leq [K':K]$.
The residue degree of any extension of a valuation from $F$ to $F'$ is bounded by $[F':F]$.
Since $K'$ embeds into the residue field of every extension to $F'$ of a valuation on $F$ trivial on $K$, it follows that any $K$-rational valuation on $F$ extends to a $K'$-rational valuation on $F'$.

Hence to prove the statement we may now assume without loss of generality that $K'=K$ and $F'=F$, so that
$F$ is the function field of the geometrically integral smooth projective curve $C$ defined over $K$.

Let $\mc{G}$ denote the absolute Galois group of $k$.
By \cite[Proposition 10.1.8]{Liu}, there exists 
a relatively minimal arithmetic surface $\mc{X}$ over $T$ whose generic fiber is $K$-isomorphic to $C$.
By
 \cite[Corollary 8.3.6]{Liu}, the special fiber $\mc{X}_s$ of $\mc{X}$ is geometrically connected and $\gen(\mc{X}_s/k)=\gen(C/K)=\gen(F/K)$.
Let $\pi:\mc{X}_s\times_k{k_{\alg}}\to \mc{X}_s$ denote the morphism of 
the base change from $k$ to $k_\alg$.
Let $\mc{C}$ and $\ovl{\mc{C}}$ denote the sets of irreducible components of $\mc{X}_s$ and $\mc{X}_s\times_k{k_{\alg}}$, respectively.
By \cite[Theorem 4.2.4]{MO15}, the morphism $\pi$ induces a bijection
 between $\mc{C}$ and the set of $\mc{G}$-orbits in $\ovl{\mc{C}}$.

The core of the argument consists in combining estimates and counting, and it relies on a careful distinction of different types of components of $\mc{X}_s$. A particular attention  needs to be paid to the components which are rational and of positive genus (so in particular singular) --- they have to be ignored in a considerate way.
To this aim we fix the following notation:
\begin{eqnarray*}
\mc{C}_\varoast & = & \{\Gamma\in\mc{C}\mid k(\Gamma)/k\mbox{ not ruled}\}\\
\ovl{\mc{C}}_\varoast & = & \{\Gamma\in\ovl{\mc{C}}\mid \pi(\Gamma)\in\mc{C}_\varoast\}\\
\ovl{\mc{C}}_{\varocircle} & = & \{\Gamma\in\ovl{\mc{C}}_\varoast\mid  \gen(\Gamma/{k_{\alg}})=0\}\\
\ovl{\mc{C}}_{\oplus} & = & \{\Gamma\in \ovl{\mc{C}}\mid \gen(\Gamma/{k_{\alg}})>\gen({k_{\alg}}(\Gamma)/{k_{\alg}})=0\}\\
\mc{C}_{\varocircle} & = & \{\pi(\Gamma)\mid\Gamma\in\ovl{\mc{C}}_{\varocircle}\}\\
\mc{C}_{\oplus} & = & \{\pi(\Gamma)\mid \Gamma\in\ovl{\mc{C}}_{\oplus}\}
\end{eqnarray*}

By \cite[Proposition 3.7]{BGVG}, every valuation in $v\in\Omega_T^{\ast}(F)$ is centred in the generic point of a unique irreducible component of $\mc{X}_s$, which we denote by $\Gamma_v$, whereby $k(\Gamma_v)=\kappa_v$ and  $\Gamma_v \in\mc{C}_\varoast$. In particular $$|\Omega^{\ast}_T(F)|=|\mc{C}_\varoast|\,.$$

Consider $v\in \Omega_T^\ast(F)$. 
We denote by $\Gamma'_v$ the normalization of~$\Gamma_v$. Note that $\Gamma'_v$ is a projective curve over $\ell_v$ which is $k$-birational  to $\Gamma_v$.
We obtain that $\kappa_v=k(\Gamma_v)=k(\Gamma'_v)=\ell_v(\Gamma'_v)$ and $\ovl{\Gamma}'_v=\Gamma'_v \times_{\ell_v}{k_{\alg}}$ is a connected smooth projective curve over ${k_{\alg}}$ such that $\Gamma_v \times_{k}{k_{\alg}}$ is $k_\alg$-birational to $\ovl{\Gamma}'^{[\ell_v:k]}_v$.
Then $\gen(\kappa_v/k)= \gen(\Gamma'_v/\ell_v)= \gen(\ovl{\Gamma}'_v/{k_{\alg}})=\gen({k_{\alg}}(\ovl{\Gamma}'_v)/{k_{\alg}})$, in view of \Cref{P:Deu-gen}.
Let $\Gamma_0$ be an irreducible component of $\Gamma_v\times_k {k_{\alg}}$.
Then $\Gamma_v$ corresponds to the ${\mc G}$-orbit $\mc{G}\Gamma_0$ in $\ovl{\mc{C}}_\varoast$,
$|\mc{G}\Gamma_0|=[\ell_v:k]$ and $\ovl{\Gamma}'_v$ is the normalization of $\Gamma_0$.
It follows that $\gen(\Gamma_0/{k_{\alg}})\geq \gen(\ovl{\Gamma}'_v/{k_{\alg}}) = \gen(\kappa_v/k)$ and $\gen(\Gamma/k_\alg)=\gen(\Gamma_0/k_\alg)$ for all $\Gamma\in \mc{G}\Gamma_0$. 
Therefore 
\begin{equation*}
[\ell_v:k]\cdot\gen(\kappa_v/k)\,\,\leq \,\, [\ell_v:k]\cdot \gen(\Gamma_0/{k_{\alg}})\,\,=\sum_{\Gamma\in \mc{G}\Gamma_0}\gen(\Gamma/{k_{\alg}})\,,
\end{equation*}
and furthermore $\Gamma_0\in\ovl{\mc{C}}_{\varocircle}\cup\ovl{\mc{C}}_{\oplus}$ if and only if $\gen(\kappa_v/k)=0$.
This shows that 
\begin{eqnarray}
\label{IE:male}
\sum_{v\in\Omega_T^\ast(F)}[\ell_v:k]\cdot \gen(\kappa_v/k) & \leq & 
\sum_{\Gamma\in \ovl{\mc C}\setminus\ovl{\mc C}_{\oplus}} \gen(\Gamma/{k_{\alg}})\quad\mbox{and}\qquad\,\\
\label{E:genzero}
|\Omega_T^{\circ}(F)|\quad  & = & \quad  |\mc{C}_{\varocircle}| + | \mc{C}_{\oplus}|\,.
\end{eqnarray}

Let $T'$ now denote a maximal unramified extension of $T$ inside an algebraic closure of $K$.
The residue field of $T'$ is an algebraic closure of $k$, which we identify with ${k_{\alg}}$.
By \Cref{L:model-unramif-ext}, $\mc{X}'=\mc{X} \times_T T'$ is an arithmetic surface over $T'$ whose special fiber $\mc{X}'_s$ is $k_\alg$-isomorphic to
$\mc{X}_s \times_k {k_{\alg}}$. 
Therefore $\mc{B}_{\mc{X}'}= \mc{B}_{\mc{X}}$.

We assume now that $\Omega_T^\ast(F)\neq\emptyset$, since otherwise there is nothing to prove.
Then $\mc{X}_s$ has an irreducible component which is not ruled.
As $C$ is geometrically connected, we have by \Cref{L:model-unramif-ext} that $\mc{X}'$ is relatively minimal over $T'$.

Since every component $\Gamma$ of $\mc{X}_s\times_k k_\alg$ is an integral curve over $k_\alg$, whereby $\gen(\Gamma/k_\alg)\geq 0$, and since $\gen(\mc{X}_s\times_k k_\alg/{k_{\alg}})  = \gen(\mc{X}_s/k)=\gen(C/K)=\gen(F/K)$,  
\Cref{L:special-betti-genus-bound} yields that  
\begin{eqnarray}\label{IE:betti-bound}
\beta(\mc{B}_{\mc{X}})  & \leq & \gen(F/K) - \sum_{\Gamma\in \ovl{\mc{C}}}\gen(\Gamma/{k_{\alg}})\,.
\end{eqnarray}

Let $\mc{W}$ be the set of $\mc{G}$-pivot vertices in $\mc{B}_{\mc X}$ 
and let $N$ be the number of $\mc{G}$-orbits in~$\mc{W}$.
By \Cref{C:pivot-orbit-bound-opt}, we have $N\leq \beta(\mc{B}_{\mc{X}})+1$.
By \Cref{P:conic-pivot},
every element of $\mc{C}_{\varocircle}$ corresponds to a $\mc G$-orbit in $\mc{W}$. 
Hence $\ovl{\mc{C}}_{\varocircle}\subseteq \mc{W}$ and 
\begin{eqnarray}\label{IE:pivot-bound}
|\mc{C}_{\varocircle}|\,\,\leq\,\, N & \leq &  \beta(\mc{B}_{\mc{X}})+1\,.
\end{eqnarray}
Since $\gen(\Gamma/k_\alg)\geq 1$ for any $\Gamma\in\ovl{\mc{C}}_{\oplus}$, we further have
\begin{eqnarray}\label{IE:C+}
|\mc{C}_\oplus|\,\,\leq \,\, |\ovl{\mc{C}}_\oplus| & \leq & \sum_{\Gamma\in\ovl{\mc{C}}_\oplus}\gen(\Gamma/k_\alg). 
\end{eqnarray}
By combining \eqref{IE:male}, \eqref{IE:betti-bound}, \eqref{IE:pivot-bound} and \eqref{IE:C+}, we obtain that 
\begin{eqnarray}\label{IE:almost}
|\Omega_T^{\circ}(F)| & \leq  & \gen(F/K)  + 1  -\!\! \sum_{\Gamma\in \ovl{\mc{C}}\setminus\ovl{\mc{C}}_\oplus}\gen(\Gamma/{k_{\alg}})\,,
\end{eqnarray}
which along with \eqref{IE:male} proves the second inequality in \eqref{IE:statement}.

Assume now that 
\begin{eqnarray}\label{E:boundattained}
|\Omega_T^\ast(F)|=\gen(F/K)+1\,.
\end{eqnarray}
Hence we have equalities on both sides in \eqref{IE:statement}. It follows that we have equality in \eqref{IE:almost}, and hence also everywhere in \eqref{IE:male}, \eqref{IE:betti-bound}, \eqref{IE:pivot-bound} and \eqref{IE:C+}.

In particular $N =  \beta(\mc{B}_{\mc{X}})+1$, which implies that $\mc{W}=\ovl{\mc{C}}_{\varocircle}$ and further, by \Cref{C:pivot-orbit-bound-opt}, that 
all $\mc{G}$-fixed vertices of $\mc{B}_{\mc{X}}$ belong to $\mc{W}$.
Therefore all $\mc{G}$-fixed vertices of $\mc{B}_{\mc{X}}$ belong to $\ovl{\mc{C}}_{\varocircle}$.

However, the equality $|\mc{C}_\varoplus|=|\ovl{\mc{C}}_\varoplus|$ in \eqref{IE:C+} implies that $\ovl{\mc{C}}_{\varoplus}$ consists of $\mc{G}$-fixed vertices of $\mc{B}_{\mc{X}}$.
Hence $\ovl{\mc{C}}_{\varoplus}=\emptyset$.

It follows that $\ovl{\mc{C}}_\varoast\setminus\ovl{\mc{C}}_\varocircle=\{\Gamma\in\ovl{\mc{C}}\mid \gen(\Gamma/k_\alg)>0\}$.
Combining \eqref{E:boundattained} with the equalities in \eqref{IE:betti-bound} and \eqref{IE:pivot-bound}, we obtain that 
\begin{eqnarray*}
|\mc{C}_{\varocircle}| + \sum_{\Gamma\in \ovl{\mc{C}}_\varoast\setminus\ovl{\mc{C}}_\varocircle}\gen(\Gamma/{k_{\alg}}) & = & \gen(F/K)+1=|\Omega_T^\ast(F)|=|\mc{C}_{\varoast}|\,.
\end{eqnarray*}
Since clearly
$$|\mc{C}_\varoast|-|\mc{C}_\varocircle|= |\mc{C}_\varoast\setminus\mc{C}_\varocircle|\leq |\ovl{\mc{C}}_\varoast\setminus\ovl{\mc{C}}_\varocircle
|\leq \sum_{\Gamma\in \ovl{\mc{C}}_\varoast\setminus\ovl{\mc{C}}_\varocircle}\gen(\Gamma/{k_{\alg}})\,,$$
we obtain that $|\mc{C}_\varoast\setminus\mc{C}_\varocircle|= |\ovl{\mc{C}}_\varoast\setminus\ovl{\mc{C}}_\varocircle|$. 
Hence $\ovl{\mc{C}}_\varoast\setminus\ovl{\mc{C}}_\varocircle$ consists of $\mc{G}$-fixed vertices of $\mc{B}_{\mc{X}}$. Since all $\mc{G}$-fixed vertices of $\mc{B}_{\mc{X}}$ belong to $\ovl{\mc{C}}_\varocircle$, we conclude that $\ovl{\mc{C}}_\varoast=\ovl{\mc{C}}_\varocircle$.
This yields that ${\mc{C}}_\varoast={\mc{C}}_\varocircle$ and hence $\Omega_T^\ast(F)=\Omega_T^\circ(F)$.

Suppose now that there exists a $k$-rational point $P$ on $\mc{X}_s$.
Then the fiber $\pi^{-1}(P)$ consists of a single point $Q$ on $\mc{X}_s\times_k k_\alg$ that is fixed by the action of $\mc{G}$.
If $Q$ is an intersection point of two components of $\mc{X}_s\times_k k_\alg$, then $Q$ is a pink vertex of $\mc{B}_{\mc{X}}$ fixed by $\mc{G}$, which contradicts the fact that all $\mc{G}$-fixed vertices of $\mc{B}_{\mc{X}}$ belong to $\ovl{\mc{C}}_{\varocircle}$ and hence are cyan.
Suppose now that $Q$ lies on a single component of $\mc{X}_s\times_k k_\alg$. 
Then $P$ lies on a single component $\Gamma$ of $\mc{X}_s$, and it follows that $\ovl{\Gamma}=\Gamma\times_k k_\alg$ is the component of $\mc{X}_s\times_k k_\alg$ containing $Q$. 
We obtain that $\ovl{\Gamma}$ is a $\mc{G}$-fixed vertex of $\mc{B}_{\mc{X}}$.
Since all $\mc{G}$-fixed vertices of $\mc{B}_{\mc{X}}$ belong to $\ovl{\mc{C}}_\varocircle$, it follows that $\gen(\ovl{\Gamma}/k_\alg)=0$, whereby $\ovl{\Gamma}$ is $k_\alg$-isomorphic to $\mathbb{P}_{k_\alg}^1$.
It follows that $\Gamma$ is a smooth curve of genus zero over $k$.
Since $\Gamma$ contains a $k$-rational point, we conclude that $\Gamma$ is rational over $k$, whereby $k(\Gamma)/k$ is ruled. This contradicts the fact that $\Gamma\in\mc{C}_\varocircle\subseteq\mc{C}_{\varoast}$.
This shows that $\mc{X}_s$ contains no $k$-rational point.

Consider now a valuation $v\in\Omega(F)$ such that $K\subseteq\mc{O}_v$ and suppose for the sake of a contradiction that $T$ is totally ramified in $\kappa_v$.
Hence there exists a unique discrete valuation ring $T'$ of $\kappa_v$ such that $T'\cap K=T$ and the 
residue field of $T'$ is equal to $k$.
The valuation $v$ is centred in point of the generic fiber of $\mc{X}$.
The closure of this point in $\mc{X}$ is a prime divisor $D$ of $\mc{X}$, and $D$ is an integral projective scheme over $T$ with function field $\kappa_v$.
Hence $T'$ is centred in a closed point $P$ of $D$.
Then $P$ is a point on the special fiber $\mc{X}_s$ and the residue field of $P$ is contained in the residue field of $T'$, which is $k$. 
This contradicts the observation above that $\mc{X}_s$ contains no $k$-rational point.
 \end{proof}

%%%%%%%%%%%%%%%%%%%%%%%%%%%%%%%%%%%%%%%%%%%%%%%%%%%%%%%%%%
%%%%%%%%%%%%%%%%%%%%%%%%%%%%%%%%%%%%%%%%%%%%%%%%%%%%%%%%%%
\bibliographystyle{plain}

\end{document}